\documentclass[12pt]{amsart}
\usepackage{amssymb,amsmath,amsthm,mathrsfs,multirow,xcolor,framed,url}
\usepackage[pdftex,
         pdfauthor={Dandan Chen and Rong Chen},
         pdftitle={On Decomposition of $\theta_2^{2n}(\tau)$ as the Sum of Lambert Series and Cusp forms},
         pdfsubject={MATHEMATICS},
         pdfkeywords={Elliptic function; Theta function; Eisenstein series.},
         pdfproducer={Latex with hyperref},
         pdfcreator={pdflatex}]{hyperref}
\newtheorem{theorem}{Theorem}[section]
\newtheorem{lemma}[theorem]{Lemma}
\newtheorem{cor}[theorem]{Corollary}

\newtheorem{prop}[theorem]{Proposition}

\theoremstyle{definition}
\newtheorem{definition}[theorem]{Definition}

\theoremstyle{remark}
\newtheorem{remark}[theorem]{Remark}

\numberwithin{equation}{section}

\newcommand\nutwid{\overset {\text{\lower 3pt\hbox{$\sim$}}}\nu}

\allowdisplaybreaks

\newcommand{\beqs}{\begin{equation*}}
\newcommand{\eeqs}{\end{equation*}}
\newcommand{\beq}{\begin{equation}}
\newcommand{\eeq}{\end{equation}}

\begin{document}
\title[On elliptic functions associated with even Dirichlet characters]
 {On a class of elliptic functions associated with the Even Dirichlet characters}

\author{Dandan Chen and Rong Chen}
\address{School of Mathematical Sciences, East China Normal University, Shanghai, People's Republic of China}
\address{Department of Mathematics, University of Florida, Gainesville, FL 32601}
\email{ddchen@stu.ecnu.edu.cn}

\address{School of Mathematical Sciences, East China Normal University, Shanghai, People's Republic of China}
\address{Department of Mathematics, University of Florida, Gainesville, FL 32601}
\email{rchen@stu.ecnu.edu.cn}

\date{\today}



\maketitle
\begin{abstract}
We construct a class of companion elliptic functions associated with the even Dirichlet characters. Using the well-known properties of the classical Weierstrass elliptic function $\wp(z|\tau)$ as the blueprint, we will derive their representations in terms of $q$-series and partial fractions. We also explore the significance of the coefficients of their power series expansions and establish the modular properties under the actions of
the arithmetic groups $\Gamma_0(N)$ and $\Gamma_1(N)$.
 \end{abstract}
 \section{Introduction}\label{section1}

We will consider a pair of companion elliptic functions generated from the twisting of the logarithmic derivative of the Jacobi theta function $\theta_1(z|\tau)$ by the even Dirichlet characters over certain subgroups of the period lattice.
We first mention some familiar properties between the Dirichlet characters and the Jacobi theta functions which can be found in standard literature.

 Let $N$ be a positive integer and $\chi$ be a Dirichlet character modulo $N$. It is extended to the set of integers $\mathbb Z$.  For all integers $m$ and $n$, it satisfies the following properties:
\begin{itemize}
\item[(1)] $\chi(1)=1$,

\item[(2)] $\chi(n+N)=\chi(n)$,

\item[(3)] $\chi(mn)=\chi(m)\chi(n)$,

\item[(4)] $\chi(n)=0$ if $\gcd (n,N)>1$.
\end{itemize}
We say $\chi$ is even if $\chi(-1)=1$ and odd if $\chi(-1)=-1$.

Let $N'$ be a positive integer which is divisible by $N$. For any character $\chi$ modulo $N$, we can form a character $\chi'$ modulo $N'$
as follows:
\begin{align*}
\chi'(k)= \begin{cases} \chi(k) \quad & \text{if}  \gcd (k,N')=1, \\[0.1in]
       0 \quad & \text{if} \gcd (k,N')>1.
\end{cases}
\end{align*}
We say that $\chi'$ is induced by the character $\chi$.

Let $\chi$ be a character modulo $N$. If there is a proper divisor $d$ of $N$ and a character modulo $d$ which induces $\chi$, then the character $\chi$ is said to be non-primitive, otherwise it is called primitive.

 Define the Gauss sum
 \begin{align*}
 g_n(\chi)=\sum_{k=1}^{N-1}\chi(k)e^{2i\pi nk/N}.
 \end{align*}
From \cite[p. 334]{AI66}, if $\chi$ is primitive, then
\begin{align*}
g_n(\chi)=\overline{\chi(n)}g_1(\chi).
\end{align*}

\begin{definition}\label{theta-defn} \emph{(Cf.\ \cite[p. 166]{HR}) Jacobi theta functions $\theta_j$  for $j = 1, 2, 3,4$ are defined as,}
\begin{align*}
\theta_1(z|\tau)&=2q^{1/8}\sum_{n=0}^\infty(-1)^nq^{n(n+1)/2}\sin(2n+1)z,
\quad\quad\theta_3(z|\tau)=1+2\sum_{n=1}^\infty q^{n^2/2}\cos2nz,\\
\theta_2(z|\tau)&=2q^{1/8}\sum_{n=0}^\infty q^{n(n+1)/2}\cos(2n+1)z,
\quad\quad\theta_4(z|\tau)=1+2\sum_{n=1}^\infty (-1)^nq^{n^2/2}\cos2nz;
\end{align*}
where $q=\exp(2\pi  i \tau)$ with  $\Im\tau>0$.
\end{definition}

The infinite product representations of theta functions are given by the following proposition.

\begin{prop} (Cf.\ \cite[p.\ 131]{Liu01}) Let $\theta_j$  for $j = 1, 2, 3, 4$ be defined as in Definition \ref{theta-defn}. Then we have
\begin{align*}
\theta_1(z|\tau)&=2q^{1/8}(\sin z)
(q;q)_\infty(qe^{-2iz} ;q)_\infty(qe^{2iz};q)_\infty,\\
\theta_2(z|\tau)&=2q^{1/8}(\cos z)
(q;q)_\infty(-qe^{-2iz};q)_\infty(-qe^{2iz};q)_\infty,\\
\theta_3(z|\tau)&=(q;q)_\infty
(-q^{1/2}e^{-2iz};q)_\infty(-q^{1/2}e^{2iz};q)_\infty,\\
\theta_4(z|\tau)&=(q;q)_\infty(q^{1/2}e^{-2iz};q)_\infty
(q^{1/2}e^{2iz};q)_\infty.
\end{align*}
\end{prop}

Here and later we use the standard $q-$series notation and $q=\exp(2\pi  i \tau)$ with  $\Im\tau>0$:
\begin{align*}
(a;q)_0=1, \quad (a;q)_n=\prod_{k=0}^{n-1}(1-aq^k), \quad (a;q)_\infty=\prod_{k=0}^\infty(1-aq^k).
\end{align*}

To motivate the main theme of the paper, we begin by recalling the  properties of the Weierstrass elliptic function $\wp(z|\tau)$ which will be served as blueprint for the rest of the work and the notation $\frac{\theta'_1}{\theta_1}(z|\tau)$ is an abbreviation for $\frac{\theta'_1(z|\tau)}{\theta_1(z|\tau)}$:

\begin{align}\label{wp-defn}
\wp(z|\tau)=\frac {1}{z^2}+\sum_{\substack{m,n=-\infty\\ (m,n)\neq(0,0)}}^{\infty}\left(\frac{1}{(z+m\pi+n\pi\tau)^2}-\frac{1}{(m\pi+n\pi\tau)^2}\right),
\end{align}
\begin{align*}
\frac{1}{\tau^2}\wp\left(\frac z\tau|-\frac 1\tau\right)=\wp(z|\tau),
\end{align*}
\begin{align}
\wp(z|\tau)&=-\left(\frac{\theta'_1}{\theta_1}\right)'(z|\tau)-\frac 13E_2(\tau)\label{theta-wp}\\
&=\csc^2z-8\sum_{n=1}^{\infty}\frac{nq^n}{1-q^n}\cos 2nz-\frac13E_2(\tau)\nonumber,
\end{align}
where
\begin{align*}
E_2(\tau)=1-24\sum_{n=1}^{\infty}\frac{nq^n}{1-q^n}.
\end{align*}
Moreover, there is  a well-known power series expansion at $z=0$ \cite[p. 11]{Apostol}:
\begin{align*}
\wp(z|\tau)=\frac{1}{z^2}+\sum_{k=1}^{\infty}(2k+1)E_{2k+2}(\tau)z^{2k};
\end{align*}
where
\begin{align*}
E_{2k}(\tau)=\frac{1}{\pi^{2k}}\sum_{\substack{m,n=-\infty\\ (m,n)\neq(0,0)}}^{\infty}\frac{1}{(m+n\tau)^{2k}}
\end{align*}
and the Eisenstein series $E_{2k}(\tau), k=2,3,4\cdots$ are modular forms of the full modular group $SL(2,\mathbb{Z})$.

Let $\chi$ be an even Dirichlet character modulo $N$. We now describe the elliptic functions referred at the begining. Define
\begin{align*}
g(z|\tau;\chi):=\sum_{k=1}^{N-1}\chi(k)\frac{\theta'_1}{\theta_1}(z+k\pi\tau|N\tau).
\end{align*}
In view of  the properties of theta functions, we show that it has a companion elliptic function:
\begin{align*}
\frac 1\tau g\left(\frac z\tau|-\frac{1}{N\tau}; \chi \right)
=\sum_{k=1}^{N-1}\chi(k)\frac{\theta'_1}{\theta_1}\left(z+\frac{k\pi}{N}|\tau\right).
\end{align*}
We will derive formulas analogues to that of Weierstrass elliptic function mentioned above. In particular, for the coefficients of power series expansion of $g(z|\tau;\chi)$ and $g\left(\frac z\tau|-\frac{1}{N\tau};\chi\right)$ at $z=0$, we will derive the Eisenstein series and Lambert series representations for these coefficients and prove that they are modular forms of the arithmetic group $\Gamma_1(N)$.

At the end, in conjunction with the following theta function identity \cite[Corollary 2]{MC-Shen}:
\begin{align}\label{four-variable}
&\frac{\theta'_1}{\theta_1}(x_1|\tau)+\frac{\theta'_1}{\theta_1}(x_2|\tau)+\frac{\theta'_1}{\theta_1}(x_3|\tau)-\frac{\theta'_1}{\theta_1}(x_1+x_2+x_3|\tau)\nonumber\\
=&\frac{\theta_1^{'}(0|\tau)\theta_1(x_1+x_2|\tau)\theta_1(x_1+x_3|\tau)\theta_1(x_2+x_3|\tau)}
{\theta_1(x_1|\tau)\theta_1(x_2|\tau)\theta_1(x_3|\tau)\theta_1(x_1+x_2+x_3|\tau)},
\end{align}
we will obtain, among other things, a set of the product representations for the Lambert series corresponding to these $q$-series for the cases $N=5, 8, 10$ and $12$.

For later use, we list the following  facts \cite[p.\ 463, p.\ 468]{ET66}:
\begin{align}
&\theta_1(z|\tau)
= iq^{1/8}e^{- iz}(q;q)_\infty(e^{2 iz};q)_\infty(qe^{-2 iz};q)_\infty\label{theta1-infty-1},\\
&\theta_1\bigg(\frac z\tau\bigg|-\frac1\tau\bigg)
=\frac 1{ i}\sqrt{- i\tau}e^{\frac{ iz^2}{\pi\tau}}\theta_1(z|\tau)\label{theta1-imagy},\\
&\theta_4\bigg(\frac z\tau\bigg|-\frac1\tau\bigg)
=\sqrt{- i\tau}e^{\frac{ iz^2}{\pi\tau}}\theta_2(z|\tau)\label{theta4-imagy},\\
&\theta'_1(0|\tau)=2q^{1/8}(q;q)^3\label{theta1-drive-0}.
\end{align}
The Dedekind eta-function is defined as
$\eta(\tau)=q^{{1}/{24}}(q;q)_\infty$.
It satisfies the imaginary transformation:
\begin{align}\label{eta-imagy}
\eta\bigg(-\frac1\tau\bigg)=\sqrt{- i\tau}\eta(\tau).
\end{align}

It is perhaps worthwhile to comment that our work is originally motivated by the identities appeared in
\cite[Eq.\ (4.8)]{Liu14} and \cite[Eq. (5.8)]{Liu07}:
\begin{align*}
&\sum_{n=1}^{\infty}\frac{q^n-q^{5n}-q^{7n}+q^{11n}}{1-q^{12n}}\sin2nz
=\frac{\eta(2\tau)\eta(4\tau)\eta^3(6\tau)\theta_1(z|3\tau)\theta_1(2z|12\tau)}{2\eta(3\tau)\eta^2(12\tau)\theta_1(z|6\tau)\theta_4(z|2\tau)},\\
&\sum_{n=1}^\infty\frac{q^n-q^{2n}-q^{3n}+q^{4n}}{1-q^{5n}}\sin 2nz
=-\frac{q^{-1/8}(q;q)^2_\infty}{2(q^5;q^5)_\infty}\frac{\theta_1(z|5\tau)\theta_1(2z|5\tau)}{\theta_1(z|\tau)},
\end{align*}
and from which we are led to the considering the following generalization of the above sums:
\begin{align*}
g(z|\tau;\chi)=4\sum_{n=1}^{\infty}\frac{\sum_{k=1}^{N-1}\chi(k)q^{kn}}{1-q^{Nn}}\sin 2nz;
\end{align*}
where $\chi$ is an even Dirichlet character modulo $N$.

Notably,  Kolberg \cite{kol}  had also investigates the function $g(z|\tau;\chi)$ and $g\left(\frac z\tau|-\frac {1}{N\tau};\chi\right)$. But the methods are different. His approach is based on the identity:
\begin{align*}
\frac{\theta^{\prime}_1}{\theta_1}\left(z|\frac{a\tau+b}{c\tau+d}\right)
=2 izc(c\tau+d)/\pi+(c\tau+d)\frac{\theta^{\prime}_1}{\theta_1}((c\tau+d)z|\tau)
\end{align*}
and one of his aims is to construct the multipliers involving the products of the Dedekind eta function of the forms $\eta^a(\tau)\eta^b(p\tau)$ so as to make, for odd prime $p$, $\eta^a(\tau)\eta^b(p\tau)E_k(\tau,\chi)$
and $\eta^a(\tau)\eta^b(p\tau)F_k(\tau,\chi)$ automorphic under the action of $\Gamma_0(p)$.
Whereas, we exploit fully the properties of Weierstrass elliptic function $\wp(z|\tau)$ and explore the elliptic aspects of $g(z|\tau;\chi)$ and
$g\left(\frac z\tau|-\frac {1}{N\tau};\chi\right)$. We also bring out their connection with the Weierstrass elliptic function $\wp(z|\tau)$ and determine precisely the transformation formulas for the corresponding Eisenstein series associated with the even character $\chi$ under the action of $\Gamma_0(N)$.

\section{Main results}\label{section2}

Let $\chi$ be a Dirichlet character modulo $N$. Define
 \begin{align*}
 Q_{\chi,N}(t)=\sum_{k=0}^{N-1}\chi(k)\frac{t^k}{1-t^N}.
 \end{align*}


 \begin{theorem}\label{g-transform}Suppose $\chi$ is an  even Dirichlet character modulo $N$. Then
\begin{align*}
g(z|\tau;\chi)
=4\sum_{n=1}^\infty\frac{\sum_{k=1}^{N-1}\chi(k)q^{kn}}{1-q^{Nn}}\sin {2nz}.
\end{align*}
\end{theorem}

\begin{proof}
The proof is based on the identity:
\begin{align}\label{theta1-deriv}
\mathrm{i}\frac{\theta'_1}{\theta_1}(z|\tau)=1+2\sum_{n\in\mathbb{Z}\setminus\{0\}}\frac{e^{2  inz}}{1-q^n},
\end{align}
where the complex number $z$ satisfies the requirement: $|q|<|e^{2  iz}|<1.$\\

From the fact that $\chi$ is even, we derive the following identity:
\begin{align*}
Q_{\chi,N}(q^{-n})&=\frac{\sum_{k=1}^{N-1}\chi(k)q^{-kn}}{1-q^{-Nn}}\\
&=-\frac{\sum_{k=1}^{N-1}\chi(k)q^{n(N-k)}}{1-q^{Nn}}\\
&=-\frac{\sum_{k=1}^{N-1}\chi(-k)q^{n(N-k)}}{1-q^{Nn}}\\
&=-\frac{\sum_{k=1}^{N-1}\chi(N-k)q^{n(N-k)}}{1-q^{Nn}}\\
&=-\frac{\sum_{k=1}^{N-1}\chi(k)q^{kn}}{1-q^{Nn}}\\
&=-Q_{\chi,N}(q^n).
\end{align*}
Replacing $\tau$ with $N\tau$ and $z$ with $z+k\pi\tau$ in \eqref{theta1-deriv}, then it becomes
\begin{align*}
 i\frac{\theta'_1}{\theta_1}\left(z+k\pi\tau|N\tau\right)
=1+2\sum_{n\in\mathbb{Z}\setminus\{0\}}e^{2 inz}\frac{q^{kn}}{1-q^{Nn}}.
\end{align*}

 Hence
 \begin{align*}
  i \sum_{k=1}^{N-1}\chi(k)\frac{\theta'_1}{\theta_1}\left(z+k\pi\tau|N\tau\right)
 &=2\sum_{n\in\mathbb{Z}\setminus\{0\}}Q_{\chi,N}(q^n)e^{2 inz}\\
 &=2\sum_{n=1}^{\infty}Q_{\chi,N}(q^n)e^{2 inz}+ 2\sum_{n=-1}^{-\infty}Q_{\chi,N}(q^n)e^{2 inz}\\
 &=2\sum_{n=1}^{\infty}Q_{\chi,N}(q^n)e^{2 inz}+ 2\sum_{n=1}^{\infty}Q_{\chi,N}(q^{-n})e^{-2 inz}\\
 &=2\sum_{n=1}^{\infty}Q_{\chi,N}(q^n)e^{2 inz}- 2\sum_{n=1}^{\infty}Q_{\chi,N}(q^{n})e^{-2 inz}\\
 &=4 i\sum_{n=1}^{\infty}Q_{\chi,N}(q^n)\sin 2nz.
 \end{align*}
 This establishes the desired identity.
 \end{proof}

 We next derive the following companion identity which can be regarded as the imaginary transformation of $g(z|\tau;\chi)$.\\

 \begin{theorem}\label{g-transform-imagy0} Suppose $\chi$ is an even Dirichlet character modulo $N$.
Then,
\begin{align*}
\frac1\tau g\bigg(\frac z\tau\bigg|-\frac 1{N\tau};\chi\bigg)
=\sum_{k=1}^{N-1}\chi(k)\cot(z-\frac{k\pi}N)+4
\sum_{n=1}^\infty\frac{g_n(\chi)q^n}{1-q^n}\sin2nz.
\end{align*}
\end{theorem}

\begin{proof}
Differentiating the imaginary transformation \eqref{theta1-imagy} logarithmically, we observe that
\begin{align*}
\frac{\theta'_1}{\theta_1}\bigg(\frac z\tau\bigg|\frac{-1}{\tau}\bigg)
=\tau \frac{\theta'_1}{\theta_1}(z|\tau)+\frac{2 iz}{\pi}.
\end{align*}

Then
\begin{align*}
g\bigg(z\bigg|\frac{-1}{N\tau};\chi\bigg)
=\tau\sum_{k=1}^{N-1}\chi(k)\frac{\theta'_1}{\theta_1}\bigg(z\tau-\frac{k\pi}{N}\bigg|\tau\bigg)
  +\frac{2 ih_{\chi}}{N}.
\end{align*}

Since $\chi$ is an even Dirichlet character modulo $N$, it is easy to verify that $h_{\chi}=0$ and we have
\begin{align*}
\frac{1}{\tau}g\bigg(\frac z\tau\bigg|\frac{-1}{N\tau};\chi\bigg)
=\sum_{k=1}^{N-1}\chi(k)\frac{\theta'_1}{\theta_1}\bigg(z-\frac{k\pi}{N}\bigg|\tau\bigg).
\end{align*}
Recall the identity (see \cite[p.\ 489]{ET66})
\begin{align*}
\frac{\theta'_1}{\theta_1}(z|\tau)
=\cot z+4\sum_{n=1}^{\infty}\frac{q^n}{1-q^n}\sin 2nz.
\end{align*}
Then
\begin{align*}
\frac{1}{\tau}g\bigg(\frac z\tau\bigg|\frac{-1}{N\tau};\chi\bigg)
=\sum_{k=1}^{N-1}\chi(k)\cot\bigg(z-\frac{k\pi}{N}\bigg)
+4\sum_{n=1}^{\infty}\frac{q^n}{1-q^n}\sum_{k=1}^{N-1}\chi(k)\sin{2n(z-\frac{k\pi}{N})};
\end{align*}
and we note
\begin{align*}
\sum_{k=1}^{N-1}\chi(k)\sin{2n(z-\frac{k\pi}{N})}
=\sum_{k=1}^{N-1}\chi(k)\cos{\frac{2kn\pi}{N}}\sin{2nz}-\sum_{k=1}^{N-1}\chi(k)\sin{\frac{2kn\pi}{N}}\cos{2nz}.
\end{align*}

Since $\chi$ is even, we have
\begin{align*}
\sum_{k=1}^{N-1}\chi(k)\sin{\frac{2kn\pi}{N}}=0
\end{align*}
and
\begin{align*}
\sum_{k=1}^{N-1}\chi(k)\cos{\frac{2kn\pi}{N}}=
\sum_{k=1}^{N-1}\chi(k)e^{i\frac{2kn\pi}{N}}=g_n(\chi).
\end{align*}
Thus, if $\chi$ is even,
\begin{align*}
\frac{1}{\tau}g\bigg(\frac z\tau\bigg|\frac{-1}{N\tau};\chi\bigg)
=\sum_{k=1}^{N-1}\chi(k)\cot\bigg(z-\frac{k\pi}{N}\bigg)+4\sum_{n=1}^{\infty}\frac{g_n(\chi)q^n}{1-q^n}\sin 2nz.
\end{align*}

\end{proof}

Appealing to the fact that
$g_n(\chi)=\overline{\chi(n)}g_1(\chi)$
when $\chi$ is primitive,
we have
 \begin{cor}\label{g-transform-imagy1} Suppose $\chi$ is an even primitive Dirichlet character modulo $N$.
Then
\begin{align*}
\frac1\tau g\bigg(\frac z\tau\bigg|-\frac 1{N\tau};\chi\bigg)
=\sum_{k=1}^{N-1}\chi(k)\cot(z-\frac{k\pi}N)
 +4g_1(\chi)\sum_{n=1}^\infty\overline {\chi(n)}\frac{q^n}{1-q^n}\sin2nz.
\end{align*}
\end{cor}

Suppose $d$ is the discriminant of the real quadratic field $\mathbb Q (\sqrt D)$. Let $\chi_d(n)=\left(\frac dn\right)$. It could be found \cite[p.\ 347]{AI66} that $\chi_d$ is an even primitive character modulo $d$
and
\begin{align*}
\sum_{k=1}^{d-1}\chi_d(k)\cos\frac{2nk\pi}{d}
&=\sqrt d\chi_d(n).
\end{align*}

\begin{cor} \label{g-transform-imagy} Suppose $d$ is the discriminant of the real quadratic field $\mathbb Q(\sqrt D)$. Let $\chi_d(n)=\left(\frac dn\right)$. Then,
we have
\begin{align*}
\frac1\tau g\bigg(\frac z\tau\bigg|-\frac 1{d\tau};\chi_d\bigg)
=\sum_{k=1}^{d-1}\chi_d(k)\cot\bigg(z-\frac{k\pi}d\bigg)+4\sqrt d \sum_{n=1}^\infty{\chi_d(n)}\frac{q^n}{1-q^n}\sin2nz.
\end{align*}
\end{cor}

To end the section, we record the following identities:\\

 \begin{cor}\label{g-cor2}
Suppose $\chi$ is an even Dirichlet character modulo $N$. Then
\begin{align*}
\sum_{k=1}^{N-1}\chi(k)\frac{\theta'_1}{\theta_1}\left(z+k\pi\tau|N\tau\right)
= 4\sum_{n=1}^\infty\frac{\sum_{k=1}^{N-1}\chi(k)q^{kn}}{1-q^{Nn}}\sin {2nz},
\end{align*}
and
\begin{align*}
\sum_{k=1}^{N-1}\chi(k)\frac{\theta'_1}{\theta_1}\bigg(z+\frac{k\pi}{N}\bigg|\tau\bigg)
=\sum_{k=1}^{N-1}\chi(k)\cot(z+\frac{k\pi}N) +4\sum_{n=1}^\infty\frac{g_n(\chi)q^n}{1-q^n}\sin2nz.
\end{align*}
 \end{cor}
 \section{Eisenstein series generated from $g\left(z|\tau;\chi\right)$ and $g\left(\frac{z}{\tau}|\frac{-1}{N\tau};\chi\right)$}\label{section3}

We begin with the investigation of the Eisenstein series generated from the elliptic function $g(z|\tau;\chi)$. We shall follow the same approach of the Weierstrass elliptic function by finding the power series expansion at $z=0$. However, for later application, we will derive a slightly more general identity which is given in the following lemma.\\

\begin{lemma}\label{Lem-wp-derive-l}Suppose $\chi$ is an even Dirichlet character modulo $N$ and integer $l\ge 1$. Then
\begin{align*}
\sum_{k=1}^{N-1}\chi(k)\wp^{(l)}(z+k\pi\tau|N\tau)
&=(-1)^l (l+1)!\sum_{m,n=-\infty}^{\infty}\frac{\chi(n)}{(z+m\pi+n\pi\tau)^{l+2}}\\
&=-\sum_{k=1}^{N-1}\chi(k)\left(\frac{\theta'_1}{\theta_1}\right)^{(l+1)}(z+k\pi\tau|N\tau).
 \end{align*}
\end{lemma}
\begin{proof}
 From \eqref{wp-defn},
 \begin{align*}
 \wp^{(l)}(z|\tau)=(-1)^l (l+1)!\sum_{m,n=-\infty}^{\infty}\frac{1}{(z+m\pi+n\pi\tau)^{l+2}}.
 \end{align*}

For integer $l\ge 1$, the series converges absolutely. Then

\begin{align*}
\sum_{k=1}^{N-1}\chi(k)\wp^{(l)}(z+k\pi\tau|N\tau)
&=(-1)^l (l+1)!\sum_{m,n=-\infty}^{\infty}\sum_{k=1}^{N-1}\frac{\chi(k)}{(z+m\pi+(nN+k)\pi\tau)^{l+2}}\\
&=(-1)^l (l+1)!\sum_{m,n=-\infty}^{\infty}\sum_{k=1}^{N-1}\frac{\chi(nN+k)}{(z+m\pi+(nN+k)\pi\tau)^{l+2}}\\
&=(-1)^l (l+1)!\sum_{m,n=-\infty}^{\infty}\frac{\chi(n)}{(z+m\pi+n\pi\tau)^{l+2}}.
\end{align*}

From \eqref{theta-wp},
\begin{align*}
\left(\frac{\theta'_1}{\theta_1}\right)'(z+k\pi\tau|N\tau)=-\wp(z+k\pi\tau|N\tau)-\frac13E_2(N\tau)
\end{align*}
 and from which
 \begin{align*}
  \sum_{k=1}^{N-1}\chi(k)\left(\frac{\theta'_1}{\theta_1}\right)^{(l+1)}(z+k\pi\tau|N\tau)
= -\sum_{k=1}^{N-1}\chi(k)\wp^{(l)}(z+k\pi\tau|N\tau),
 \end{align*}
we complete the proof of Lemma \ref{Lem-wp-derive-l}.
\end{proof}

\begin{lemma}\label{Lem-wp-derive-2l(2l+1)-cos(sin)}Suppose $\chi$ is an even Dirichlet character modulo $N$  and integer $l\ge 0$. Then
\begin{align*}
\sum_{k=1}^{N-1}\chi(k)\wp^{(2l)}(z+k\pi\tau|N\tau)
=(-1)^{l+1}2^{2l+3}\sum_{n=1}^{\infty}\frac{n^{2l+1}\sum_{k=1}^{N-1}\chi(k)q^{kn}}{1-q^{Nn}}\cos 2nz
\end{align*}
and
\begin{align*}
\sum_{k=1}^{N-1}\chi(k)\wp^{(2l+1)}(z+k\pi\tau|N\tau)
=(-1)^{l}2^{2l+4}\sum_{n=1}^{\infty}\frac{n^{2l+2}\sum_{k=1}^{N-1}\chi(k)q^{kn}}{1-q^{Nn}}\sin 2nz.
\end{align*}
\end{lemma}

\begin{proof}
From Corollary \ref{g-cor2}, we have
\begin{align*}
\sum_{k=1}^{N-1}\chi(k)\frac{\theta'_1}{\theta_1}\left(z+k\pi\tau|N\tau\right)= 4\sum_{n=1}^\infty\frac{\sum_{k=1}^{N-1}\chi(k)q^{kn}}{1-q^{Nn}}\sin {2nz}.
\end{align*}
Thus
\begin{align}\label{g-cor2-1-derive}
\sum_{k=1}^{N-1}\chi(k)\left(\frac{\theta'_1}{\theta_1}\right)'\left(z+k\pi\tau|N\tau\right)
= 8\sum_{n=1}^\infty \frac{n\sum_{k=1}^{N-1}\chi(k)q^{kn}}{1-q^{Nn}}\cos {2nz}.
\end{align}\\

From \eqref{theta-wp}, we find
\begin{align*}
\left(\frac{\theta'_1}{\theta_1}\right)'(z+k\pi\tau|N\tau)=-\wp(z+k\pi\tau|N\tau)-\frac13E_2(N\tau).
\end{align*}
Then, together with \eqref{g-cor2-1-derive}, we obtain
\begin{align*}
\sum_{k=1}^{N-1}\chi(k)\wp(z+k\pi\tau|N\tau)
=-8\sum_{n=1}^{\infty}\frac{n\sum_{k=1}^{N-1}\chi(k)q^{kn}}{1-q^{Nn}}\cos 2nz.
\end{align*}
Differentiating the above identity $2l$ times, we obtain the desired result. \\

We omit the proof of the second identity, since it is identical.\\
\end{proof}

\begin{lemma}
Suppose $\chi$ is a Dirichlet character modulo $N$. Then, formally,
\begin{align*}
\sum_{n=1}^\infty \frac{n^l\sum_{k=1}^{N-1}\chi(k)q^{kn}}{1-q^{Nn}}f(n)=\sum_{m,n=1}^{\infty}n^lf(n)\chi(m)q^{mn}.
\end{align*}
\end{lemma}
\begin{proof}
We note
\begin{align*}
\frac{\sum_{k=1}^{N-1}\chi(k)q^{kn}}{1-q^{Nn}}
&=\sum_{k=1}^{N-1}\chi(k)q^{kn}\sum_{j=0}^{\infty}q^{jNn}\\
&=\sum_{j=0}^{\infty}\sum_{k=1}^{N-1}\chi(jN+k)q^{(jN+k)n}\\
&=\sum_{m=1}^{\infty}\chi(m)q^{mn}
\end{align*}
and the desired identity follows readily.\\
\end{proof}

In particular, choosing $f(n)=\cos 2nz$ and $\sin 2nz$, respectively, we derive from Lemmas \ref{Lem-wp-derive-l} and  \ref{Lem-wp-derive-2l(2l+1)-cos(sin)}, the following identities,

\begin{lemma}\label{Lem-wp-derive-2l(2l+1)-E} Suppose $\chi$ is an even Dirichlet character modulo $N$. Then
\begin{itemize}
\item[(1).] for integer $l\ge 1$,
\begin{align*}
\sum_{k=1}^{N-1}\chi(k)\wp^{(2l)}(z+k\pi\tau|N\tau)
&=(2l+1)!\sum_{m,n=-\infty}^{\infty}\frac{\chi(n)}{(z+m\pi+n\pi\tau)^{2l+2}}\\
&=(-1)^{l+1}2^{2l+3}\sum_{m,n=1}^{\infty}n^{2l+1}\chi(m) q^{mn}\cos 2nz;
\end{align*}
\item[(2).] for integer $l\ge 0$,
\begin{align*}
\sum_{k=1}^{N-1}\chi(k)\wp^{(2l+1)}(z+k\pi\tau|N\tau)
&=-(2l+2)!\sum_{m,n=-\infty}^{\infty}\frac{\chi(n)}{(z+m\pi+n\pi\tau)^{2l+3}}\\
&=(-1)^{l}2^{2l+4}\sum_{m,n=1}^{\infty}n^{2l+2}\chi(m)q^{mn}\sin 2nz.
\end{align*}
\end{itemize}
\end{lemma}

We define, for integer $k\ge 1$,
\begin{align*}
E_{2k}(\tau,\chi)
=\frac{(-1)^k2^{2k+1}}{(2k-1)!}\sum_{m,n=1}^{\infty}n^{2k-1}\chi(m)q^{mn}.
\end{align*}

We derive the power series expansion of $g(z|\tau;\chi)$.

\begin{theorem} Suppose $\chi$ is an even Dirichlet character modulo $N$. Then
\begin{align*}
&g(z|\tau;\chi)
=-\sum_{k=0}^{\infty}E_{2k+2}(\tau,\chi)z^{2k+1},\\
&\sum_{k=0}^{N-1}\chi(k)\wp(z+k\pi\tau|N\tau)
=\sum_{k=0}^{\infty}(2k+1) E_{2k+2}(\tau,\chi) z^{2k}
\end{align*}
and, for integer $k\ge 2$,
\begin{align*}
E_{2k}(\tau,\chi)=\frac{1}{\pi^{2k}}\sum_{m,n=-\infty}^{\infty}\frac{\chi(n)}{(m+n\tau)^{2k}}.
\end{align*}
\end{theorem}

\begin{proof}
Since $g(z|\tau;\chi)$ is an odd function of $z$, the power series expansion at $z=0$ is of the form
\begin{align*}
g(z|\tau;\chi)=\sum_{k=0}^{\infty}\frac{a_{2k+1}}{(2k+1)!}z^{2k+1}.
\end{align*}
Since
\begin{align*}
g'(z|\tau;\chi)=-\sum_{k=1}^{N-1}\chi(k)\wp(z+k\pi\tau|N\tau),
\end{align*}
we have
\begin{align*}
g^{(2l+1)}(z|\tau;\chi)=-\sum_{k=1}^{N-1}\chi(k)\wp^{(2l)}(z+k\pi\tau|N\tau).
\end{align*}
From Lemma \ref{Lem-wp-derive-2l(2l+1)-E}, for integer $l\ge 0$,
\begin{align*}
a_{2l+1}=g^{(2l+1)}(0|\tau;\chi)=-(2l+1)!E_{2l+2}(\tau,\chi)
\end{align*}
and for integer $l\ge 1$,
\begin{align*}
E_{2l+2}(\tau,\chi)=\sum_{m,n=-\infty}^{\infty}\frac{\chi(n)}{(m\pi+n\pi\tau)^{2l+2}}.
\end{align*}
\end{proof}

Letting  $z=\frac{\pi}{2}, \frac{\pi}{3},\frac{\pi}{4}$ in Lemma \ref{Lem-wp-derive-2l(2l+1)-E}, we obtain additional identities:

\begin{align*}
\sum_{m,n=-\infty}^{\infty}\frac{\chi(n)}{(1/2+m+n\tau)^{2l+2}}
=\frac{(-1)^{l+1}2^{2l+3}\pi^{2l+2}}{(2l+1)!}\sum_{m,n=1}^{\infty}(-1)^nn^{2l+1}\chi(m)q^{mn}
\end{align*}
for integer $l\ge 1$;

\begin{align*}
\sum_{m,n=-\infty}^{\infty}\frac{\chi(n)}{(1/3+m+n\tau)^{2l+3}}
=\frac{(-1)^{l+1}2^{2l+4}\pi^{2l+3}}{(2l+2)!}\sum_{m,n=1}^{\infty}n^{2l+2}\left(\frac{-3}{n}\right)\chi(m)q^{mn}
\end{align*}
and
\begin{align*}
\sum_{m,n=-\infty}^{\infty}\frac{\chi(n)}{(1/4+m+n\tau)^{2l+3}}
=\frac{(-1)^{l+1}2^{2l+4}\pi^{2l+3}}{(2l+2)!}\sum_{m,n=1}^{\infty}n^{2l+2}\left(\frac {-4}{n}\right)\chi(m)q^{mn}
\end{align*}
for integer $l\ge 0$.


Next, we investigate the Eisenstein series generated from the elliptic function $g\left(\frac{z}{\tau}|\frac{-1}{N\tau};\chi\right)$. \\

Suppose $\chi$ is an even Dirichlet character modulo $N$.
We recall that
\begin{align*}
\frac1\tau g\bigg(\frac z\tau\bigg|-\frac 1{N\tau};\chi\bigg)
&=\sum_{k=1}^{N-1}\chi(k)\cot(z+\frac{k\pi}N)+4\sum_{n=1}^\infty\frac{g_n(\chi)q^n}{1-q^n}\sin2nz\\
&=\sum_{k=1}^{N-1}\chi(k)\frac{\theta'_1}{\theta_1}\bigg(z+\frac{k\pi}{N}\bigg|\tau\bigg).
\end{align*}

Let
\begin{align*}
\sum_{k=1}^{N-1}\chi(k)\cot(z+\frac{k\pi}N)
=\sum_{n=0}^{\infty} B_n(\chi) z^{n}.
\end{align*}
Since  $g\left(\frac{z}{\tau}|\frac{-1}{N\tau};\chi\right)$ is an odd function of $z$, we have $B_{2n}(\chi)=0$.

Note that
\begin{align*}
\sum_{n=1}^\infty\frac{g_n(\chi)q^n}{1-q^n}\sin2nz
=\sum_{m,n=1}^{\infty}g_n(\chi)q^{mn}\sin 2nz.
\end{align*}
Then
\begin{align}\label{g-transform-imagy-z-B(exp)}
\frac{1}{\tau}g\left(\frac{z}{\tau}|\frac{-1}{N\tau};\chi\right)
=\sum_{n=0}^{\infty} B_{2n+1} z^{2n+1}+4\sum_{m,n=1}^{\infty}g_n(\chi)q^{mn}\sin 2nz.
\end{align}\\

We define, for integer $k\ge 0$,

\begin{align*}
F_{2k+2}(\tau,\chi):=- B_{2k+1}(\chi)+\frac{(-1)^{k+1}2^{2k+1}}{(2k+1)!}
\sum_{m,n=1}^{\infty}n^{2k+1}g_n(\chi)q^{mn}.
\end{align*}

We  now derive the power series expansion of
 $g\left(\frac{z}{\tau}|\frac{-1}{N\tau};\chi\right)$.

\begin{theorem}\label{g-transform-z-E(exp)} Suppose $\chi$ is an even Dirichlet character modulo $N$. Then
\begin{align*}
&\frac{1}{\tau} g\left(\frac{z}{\tau}|\frac{-1}{N\tau};\chi\right)
=-\sum_{k=0}^{\infty}F_{2k+2}(\tau,\chi)z^{2k+1},\\
&\sum_{k=1}^{N-1}\chi(k)\wp\left(z+\frac{k\pi}{N}|\tau\right)
=\sum_{k=0}^{\infty}(2k+1) F_{2k+2}(\tau,\chi) z^{2k}
\end{align*}
and, for integer $k\ge 1$,
\begin{align*}
F_{2k+2}(\tau,\chi)=\frac{1}{\pi^{2k+2}}\sum_{m,n=-\infty}^{\infty}\frac{\chi(m)}{\left({m}/{N}+n\tau\right)^{2k+2}}.
\end{align*}
\end{theorem}

\begin{proof}
Since  $g\left(\frac{z}{\tau}|\frac{-1}{N\tau};\chi\right)$ is an odd function of $z$, the power series expansion at $z=0$ is of the form
\begin{align*}
\frac{1}{\tau} g\left(\frac{z}{\tau}|\frac{-1}{N\tau};\chi\right)
=\sum_{k=0}^{\infty}\frac{d_{2k+1}}{(2k+1)!}z^{2k+1}.
\end{align*}
From \eqref{g-transform-imagy-z-B(exp)}, for integer $l\ge 0$,
\begin{align*}
d_{2l+1}
&=(2l+1)! B_{2l+1}(\chi)+(-1)^l 2^{2l+3}\sum_{m,n=1}^{\infty}n^{2l+1}g_n(\chi)q^{mn}\\
&=-(2l+1)!F_{2l+2}(\tau,\chi).
\end{align*}\\

On the other hand, from Theorem \ref{g-transform-z-E(exp)},
\begin{align*}
\frac{1}{\tau} g\left(\frac{z}{\tau}|\frac{-1}{N\tau};\chi\right)
=-\sum_{k=0}^{\infty}\frac{E_{2k+2}\left(\frac{-1}{N\tau}, \chi\right)}{\tau^{2k+2}}z^{2k+1}
\end{align*}
and, for integer $l\ge 1$,
\begin{align*}
\frac{1}{\tau^{2l+2}}E_{2l+2}\left(\frac{-1}{N\tau}, \chi\right)
=N^{2l+2}\sum_{m,n=-\infty}^{\infty}\frac{\chi(m)}{(m\pi+nN\pi\tau)^{2l+2}}.
\end{align*}
Then, for integer $l\ge 0$, we have
\begin{align*}
d_{2l+1}
=-\frac{(2l+1)!}{\tau^{2l+2}}E_{2l+2}\left(\frac{-1}{N\tau}, \chi\right)
\end{align*}
and
\begin{align*}
F_{2l+2}(\tau,\chi)=\frac{1}{\tau^{2l+2}}E_{2l+2}\left(\frac{-1}{N\tau}, \chi\right).
\end{align*}
\end{proof}

For future references, we record without proving the following identities which are readily derivable from the earlier identities. \\

Suppose $\chi$ is an even Dirichlet character modulo $N$. Then
\begin{itemize}
\item[(1)]
For integer $l\ge 1$,
\begin{align*}
\sum_{k=1}^{N-1}\chi(k)\wp^{(l)}\left(z+\frac{k\pi}{N}|\tau\right)
&=(-1)^l (l+1)!N^{l+2}\sum_{m,n=-\infty}^{\infty}\frac{\chi(m)}{(Nz+m\pi+nN\pi\tau)^{l+2}}\\
&=-\sum_{k=1}^{N-1}\chi(k)\left(\frac{\theta'_1}{\theta_1}\right)^{(l+1)}\left(z+\frac{k\pi}{N}|\tau\right);
\end{align*}
\item[(2)]
For integer $l\ge 0$,
\begin{align*}
\sum_{k=1}^{N-1}\chi(k)\wp^{(2l)}\left(z+\frac{k\pi}{N}|\tau\right)
=&-\sum_{k=1}^{N-1}\chi(k)\cot^{(2l+1)}(z+\frac{k\pi}{N})\\
&+(-1)^{l+1}2^{2l+3}\sum_{m,n=1}^\infty n^{2l+1}g_n(\chi)q^{mn}\cos 2nz
\end{align*}
and
\begin{align*}
\sum_{k=1}^{N-1}\chi(k)\wp^{(2l+1)}\left(z+\frac{k\pi}{N}|\tau\right)
=&-\sum_{k=1}^{N-1}\chi(k)\cot^{(2l+2)}(z+\frac{k\pi}{N})\\
&+(-1)^{l}2^{2l+4}\sum_{m,n=1}^{\infty} n^{2l+2}g_n(\chi)q^{mn}\sin 2nz.
\end{align*}
\end{itemize}
Taking $q\rightarrow0$ in the above two equation, we can obtain
\begin{cor} For positive integer $l$, we have
\begin{align*}
\sum_{m=-\infty}^{+\infty}\frac{\chi(m)}{(Nz+m\pi)^{l+1}}
=\frac{1}{l!N^{l+1}}\sum_{k=1}^{N-1}\chi(k)\cot^{(l)}\left(z+\frac{k\pi}{N}\right).
\end{align*}
Especially, when $z=0$, we have
\begin{align*}
\sum_{m=-\infty}^{+\infty}\frac{\chi(m)}{m^{l+1}}
=\frac {1}{l!}\left(\frac{\pi}{N}\right)^{l+1}\sum_{k=1}^{N-1}\chi(k)\cot^{(l)}\left(\frac{k\pi}{N}\right).
\end{align*}
\end{cor}

Hence,
\begin{align*}
B_{l}(\chi)&=\sum_{k=1}^{N-1}\frac{\chi(k)}{l!}\cot^{(l)}\left(\frac{k\pi}{N}\right)\\
&=\left(\frac{N}{\pi}\right)^{l+1}\sum_{m=-\infty}^{+\infty}\frac{\chi(m)}{m^{l+1}}.
\end{align*}
\section{Modular forms generated from $g(z|\tau;\chi)$
and $g\left(\frac{z}{\tau}|\frac{-1}{N\tau};\chi\right)$. }\label{section4}

Recall,  for integer $k\ge 2$,
\begin{align*}
E_{2k}(\tau,\chi)=\frac{1}{\pi^{2k}}\sum_{m,n=-\infty}^{\infty}\frac{\chi(n)}{(m+n\tau)^{2k}}
\end{align*}
and
\begin{align*}
F_{2k}(\tau,\chi)=\frac{N^{2k}}{\pi^{2k}}\sum_{m,n=-\infty}^{\infty}\frac{\chi(m)}{\left(m+nN\tau\right)^{2k}}.
\end{align*}
The relation between them is given by
\begin{align*}
E_{2l+2}\left(\frac{-1}{N\tau}, \chi\right)={\tau^{2l+2}}F_{2l+2}(\tau,\chi).
\end{align*}

\begin{theorem} Suppose integer $k\ge 2$.
\begin{itemize}
\item[(1)]If
\begin{align*}
\left(\begin{array}{cc}
a&b\\
c&d
\end{array}\right)\in \Gamma_0(N),
\end{align*}
then
\begin{align*}
E_{2k}\left(\frac{a\tau+b}{c\tau+d},\chi\right)=\overline{\chi(a)}(c\tau+d)^{2k}E_{2k}(\tau,\chi);
\end{align*}
\item[(2)]If
\begin{align*}
\left(\begin{array}{cc}
a&b\\
c&d
\end{array}\right)\in \Gamma_1(N).
\end{align*}
then
\begin{align*}
E_{2k}\left(\frac{a\tau+b}{c\tau+d},\chi\right)=(c\tau+d)^{2k}E_{2k}(\tau,\chi).
\end{align*}
\end{itemize}
\end{theorem}

\begin{proof}
Since $c\equiv 0\pmod N$, we have
\begin{align*}
\chi(an+cm)=\chi(a)\chi(n).
\end{align*}
Then
\begin{align*}
\pi^{2k}(c\tau+d)^{-2k}E_{2k+2}\left(\frac{a\tau+b}{c\tau+d},\chi\right)
&=\sum_{m,n=-\infty}^{+\infty}\frac{\chi(n)}{(m(c\tau+d)+n(a\tau+b))^{2k}}\\
&=\overline{\chi(a)}\sum_{m,n=-\infty}^{+\infty}\frac{\chi(an+mc)}{((bn+dm)+(an+mc)\tau)^{2k}}\\
&=\overline{\chi(a)}\sum_{m,n=-\infty}^{+\infty}\frac{\chi(m)}{(m\tau+n)^{2k}}.
\end{align*}
For the last equality, we appeal to the fact that
\begin{align*}
\left(\begin{array}{cc}
a&c\\
b&d
\end{array}\right)
\left(\begin{array}{c}
n\\
m
\end{array}\right)
=\left(\begin{array}{c}
an+cm\\
bn+dm
\end{array}\right)
\end{align*}
and
\begin{align*}
\left(\begin{array}{cc}
a&c\\
b&d
\end{array}\right): \mathbb Z^2\to \mathbb Z^2
\end{align*}
is a 1-1 and onto map; where $\mathbb Z$ denotes the set of rational integers and
\begin{align*}
\mathbb Z^2=\left\{\left(\begin{array}{c}
n\\
m
\end{array}\right):n,m\in\mathbb Z\right\}
\end{align*}
\end{proof}

Similarly,

\begin{theorem} Suppose integer $k\ge 2$.
\begin{itemize}
\item[(1)]If
\begin{align*}
\left(\begin{array}{cc}
a&b\\
c&d
\end{array}\right)\in \Gamma_0(N),
\end{align*}
then
\begin{align*}
F_{2k}\left(\frac{a\tau+b}{c\tau+d},\chi\right)=\chi(a)(c\tau+d)^{2k}F_{2k}(\tau,\chi);
\end{align*}
\item[(2)]If
\begin{align*}
\left(\begin{array}{cc}
a&b\\
c&d
\end{array}\right)\in \Gamma_1(N).
\end{align*}
then
\begin{align*}
F_{2k}\left(\frac{a\tau+b}{c\tau+d},\chi\right)=(c\tau+d)^{2k}F_{2k}(\tau,\chi).
\end{align*}
\end{itemize}
\end{theorem}

\begin{proof}
Since $c\equiv 0\pmod N$ and $ad-bc=1$, we have
\begin{align*}
\chi(a)\chi(d)=\chi(ad-bc)=1, \quad\quad \chi(bNn+dm)=\chi(d)\chi(m)
\end{align*}
and
\begin{align*}
\chi(m)=\chi(a)\chi(bNn+dm).
\end{align*}
Then
\begin{align*}
\pi^{2k}N^{-2k}(c\tau+d)^{-2k}F_{2k+2}\left(\frac{a\tau+b}{c\tau+d},\chi\right)
&=\sum_{m,n=-\infty}^{+\infty}\frac{\chi(m)}{(m(c\tau+d)+nN(a\tau+b))^{2k}}\\
&=\chi(a)\sum_{m,n=-\infty}^{+\infty}\frac{\chi(bNn+dm)}{((bNn+dm)+(an+cN^{-1}m)N\tau)^{2k}}\\
&=\chi(a)\sum_{m,n=-\infty}^{+\infty}\frac{\chi(m)}{(m\tau+n)^{2k}}.
\end{align*}
The last equality follows from the facts:
\begin{align*}
\left(\begin{array}{cc}
a&cN^{-1}\\
bN&d
\end{array}\right)
\left(\begin{array}{c}
n\\
m
\end{array}\right)
=\left(\begin{array}{c}
an+cN^{-1}m\\
bNn+dm
\end{array}\right)
\end{align*}
and
\begin{align*}
\left(\begin{array}{cc}
a&cN^{-1}\\
bN&d
\end{array}\right): \mathbb Z^2\to \mathbb Z^2
\end{align*}
is a 1-1 and onto map.
\end{proof}
\section{Some Lambert series with product representations}\label{section5}

We note that there are four moduli $N=5, 8, 10$ and $12$ in which there are four reduced residues modulo $N$. Using \eqref{four-variable}, Theorem \ref{g-transform} and Corollary \ref{g-cor2} in conjunction with the theta identity \eqref{theta1-infty-1}, we will establish several identities with parameter $z$. These identities will generate interesting Lambert series after specializing the choices of $z$. \\

We will present  the cases $N=8$ and 12 first, since the characters involved  are primitive.
The case for $N=10$ is slightly more complicated, since character involved is not primitive and it will be presented next.
The case for $N=5$ has already appeared in the literature and we will omit the details.
It is worth  reminding the reader that $q=\exp(2\pi  i \tau)$ with $\Im\tau>0$ and the well-known identity:
\begin{align*}
\sum_{n=1}^{\infty}\frac{\sum_{k=1}^{N-1}\chi(k)nq^{kn}}{1-q^{Nn}}=\sum_{n=1}^{\infty}\frac{\chi(n)q^n}{(1-q^n)^2}.
\end{align*}

\subsection{Representations for d=8}
\begin{cor} \label{cor-Lamber-8}There holds the identity
\begin{align} \label{Lamber-8}
\sum_{n=1}^\infty\frac{q^n-q^{3n}-q^{5n}+q^{7n}}{1-q^{8n}}\sin 2nz
=\frac{(q^4;q^4)_\infty(q^2;q^2)_\infty^2}{2(q^8;q^8)_\infty}\frac{\theta_1(2z|8\tau)}{\theta_4(z|2\tau)}.
\end{align}
\end{cor}
\begin{proof} Taking $N=8$ in Theorem \ref{g-transform} and $\chi(n)=\chi_8(n)=\left(\frac{8}{n}\right)$, we obtain
\begin{align}\label{eq:4.5.3}
&4\sum_{n=1}^\infty\frac{q^n-q^{3n}-q^{5n}+q^{7n}}{1-q^{8n}}\sin2nz\nonumber\\
=&\frac{\theta'_1}{\theta_1}(z-\pi\tau|8\tau)+\frac{\theta'_1}{\theta_1}(z+\pi\tau|8\tau)
-\frac{\theta'_1}{\theta_1}(z-3\pi\tau|8\tau)-\frac{\theta'_1}{\theta_1}(z+3\pi\tau|8\tau).
\end{align}
Replacing $\tau$ by $8\tau$ and letting $x_1=z-\pi\tau$ , $x_2=z+\pi\tau$ and $x_3=-z+3\pi\tau$ in \eqref{four-variable}, we obtain
 \begin{align*}
&\frac{\theta'_1}{\theta_1}(z-\pi\tau|8\tau)+\frac{\theta'_1}{\theta_1}(z+\pi\tau|8\tau)
-\frac{\theta'_1}{\theta_1}(z-3\pi\tau|8\tau)-\frac{\theta'_1}{\theta_1}(z+3\pi\tau|8\tau)\\
=&-\frac{\theta_1^{'}(0|8\tau)\theta_1(2z|8\tau)\theta_1(2\pi\tau|8\tau)\theta_1(4\pi\tau|8\tau)}
{\theta_1(z-\pi\tau|8\tau)\theta_1(z+\pi\tau|8\tau)\theta_1(z-3\pi\tau|8\tau)\theta_1(z+3\pi\tau|8\tau)}.
\end{align*}
Replacing $\tau$ by $8\tau$ in \eqref{theta1-drive-0}, we find that
\begin{align}
\theta'_1(0|8\tau)=2q(q^{8};q^{8})^3.
\end{align}

From \eqref{theta1-infty-1} by some elementary calculation, it is easy to show that
\begin{align}\label{eq:4.6}
\theta_1(2\pi\tau|8\tau)\theta_1(4\pi\tau|8\tau)=-q^{-1}(q^2;q^2)_\infty(q^4;q^4)_\infty,
\end{align}
and
\begin{align}\label{eq:4.7}
\theta_1(z-\pi\tau|8\tau)\theta_1(z+\pi\tau|8\tau)\theta_1(z-3\pi\tau|8\tau)\theta_1(z+3\pi\tau|8\tau)
=\frac{(q^{8};q^{8})^4_\infty\theta_4(z|2\tau)}{(q^2;q^2)_\infty}.
\end{align}
Combining the above equation \eqref{eq:4.5.3}-\eqref{eq:4.7}, we obtain  \eqref{Lamber-8}.\\
\end{proof}

By \eqref{Lamber-8} we find
\begin{itemize}
\item[(1)]
Dividing both sides of \eqref{Lamber-8} by $z$ and then letting $z\rightarrow 0$, we are led to
\begin{align*}
\sum_{n=1}^\infty\frac{n(q^n-q^{3n}-q^{5n}+q^{7n})}{1-q^{8n}}
=\frac{q(q^2;q^2)^3_\infty (q^4;q^4)_\infty (q^8;q^8)^2_\infty}{(q;q)^2_\infty}.
\end{align*}

\item[(2)] Taking $z=\frac{\pi}{3}$ and $z=\frac{\pi}{4}$ in  \eqref{Lamber-8}, respectively, we obtain  identities
\begin{align*}
&\sum_{n=1}^\infty\bigg(\frac{n}{3}\bigg)\frac{q^n-q^{3n}-q^{5n}+q^{7n}}{1-q^{8n}}
=\frac{q(q;q)_\infty(q^4;q^4)_\infty(q^{6};q^{6})_\infty(q^{24};q^{24})_\infty}{(q^3;q^3)_\infty(q^8;q^8)_\infty},\\
&\sum_{n=1}^\infty\bigg(\frac{-4}{n}\bigg)\frac{q^n-q^{3n}-q^{5n}+q^{7n}}{1-q^{8n}}
=\frac{q(q^2;q^2)_\infty^2(q^{16};q^{16})_\infty^2}{(q^4;q^4)_\infty(q^8;q^8)_\infty}.
\end{align*}
\end{itemize}

\begin{cor} \label{cor-Lamber-imagy-8} There holds the identity
\begin{align}\label{Lamber-imagy-8}
\frac{\sin{2z}}{\cos4z}-2\sum_{n=1}^\infty\chi_8(n)\frac{q^n}{1-q^{n}}\sin 2nz
=q^{3/8}\frac{(q^4;q^4)^2_\infty(q^2;q^2)_\infty}{(q;q)_\infty}\frac{\theta_1(2z|\tau)}{\theta_2(4z|4\tau)}.
\end{align}
\end{cor}

\begin{proof} We first note that, in terms of the Dedekind $\eta-$ function, the identity \eqref{Lamber-8} can be expressed as
\begin{align}\label{g-8}
g(z|\tau;\chi_8)=\frac{\eta(4\tau)\eta^2(2\tau)}{2\eta(8\tau)}\frac{\theta_1(2z|8\tau)}{\theta_4(z|2\tau)}.
\end{align}

From Corollary \ref{g-transform-imagy}, we have
\begin{align*}
\frac{1}{\tau}g\bigg(\frac{z}{\tau}\bigg|-\frac{1}{8\tau};\chi_8\bigg)
=\sum_{k=1}^7\chi_8(k)\cot\bigg(z-\frac{k\pi}{8}\bigg)
+4\sqrt8\sum_{n=1}^{\infty}\chi_8(n)\frac{q^n}{1-q^n}\sin2nz.
\end{align*}
From  \eqref{theta1-imagy},\eqref{theta4-imagy} and \eqref{eta-imagy}, we deduce that
\begin{align*}
&\theta_1\bigg(\frac{2z}{\tau}\bigg|-\frac{1}{\tau}\bigg)
=- i\sqrt{- i\tau}e^{\frac{4 iz^2}{\pi\tau}}\theta_1(2z|\tau),
\quad\quad \eta\bigg(-\frac{1}{2\tau}\bigg)=\sqrt{-2 i\tau}\eta(2\tau),\\
&\theta_4\bigg(\frac{z}{\tau}\bigg|-\frac{1}{4\tau}\bigg)
=\sqrt{-4 i\tau}e^{\frac{4 iz^2}{\pi\tau}}\theta_2(4z|4\tau),
\quad\quad\eta\bigg(-\frac{1}{4\tau}\bigg)=\sqrt{-4 i\tau}\eta(4\tau).
\end{align*}
Substituting them into the right-hand side of \eqref{Lamber-8} and replacing $\tau$ by $-\frac{1}{8\tau}$ and $z$ by $\frac{z}{\tau}$ in \eqref{g-8},  we obtain
\begin{align*}
\sum_{k=1}^7\chi_8(k)\cot\bigg(z-\frac{k\pi}{8}\bigg)
+4\sqrt8\sum_{n=1}^{\infty}\chi_8(n)\frac{q^n}{1-q^n}\sin2nz=-4\sqrt2\frac{\eta(2\tau)\eta^2(4\tau)\theta_1(2z|\tau)}{\eta(\tau)\theta_2(4z|4\tau)}.
\end{align*}
Taking $q\rightarrow0$ in the above equation, we have
\begin{align*}
\sum_{k=1}^7\chi_8(k)\cot\bigg(z-\frac{k\pi}{8}\bigg)
=-4\sqrt2\frac{\sin 2z}{\cos 4z}.
\end{align*}
Combining the above two equations, we complete the proof of  \eqref{Lamber-imagy-8}.

\end{proof}

By \eqref{Lamber-imagy-8} we have
\begin{itemize}
\item[(1)]
Dividing both sides of \eqref{Lamber-imagy-8} by $z$ and then letting $z\rightarrow 0$, one can easily show that

\begin{align*}
1-2\sum_{n=1}^\infty\chi_8(n)\frac{nq^n}{1-q^n}
=\frac{(q;q)^2_\infty(q^2;q^2)_\infty(q^4;q^4)^3_\infty}{(q^8;q^8)^2_\infty}.
\end{align*}

\item[(2)] Taking $z=\frac\pi3$ and $z=\frac\pi4 $ in  \eqref{Lamber-imagy-8}, respectively, we obtain identities
\begin{align*}
&1+\sum_{n=1}^{\infty}\bigg(\frac{-24}{n}\bigg)\frac{q^n}{1-q^{n}}
=\frac{(q^2;q^2)_\infty(q^3;q^3)_\infty(q^8;q^8)_\infty(q^{12};q^{12})_\infty}{(q;q)_\infty(q^{24};q^{24})_\infty},\\
&1+2\sum_{n=1}^{\infty}\bigg(\frac{-32}n\bigg)\frac{q^n}{1-q^{n}}
=\frac{(q^2;q^2)^3_\infty(q^4;q^4)^3_\infty}{(q;q)^2_\infty(q^8;q^8)^2_\infty}.
\end{align*}
\end{itemize}


\subsection{Representations for d=10}
Define
\begin{align*}
\chi(n)=\psi_{10}(n):=\begin{cases}1&\text{if}~ n\equiv1,9\pmod{10},\\
-1&\text{if} ~n\equiv3,7\pmod{10}.
\end{cases}
\end{align*}
We note the character $\psi$ is induced from the Kronecker symbol $\chi_5(n)=\left(\frac{n}{5}\right)$, it is not primitive.
\begin{cor}\label{cor-Lamber-10}There holds the identity
\begin{align}\label{Lamber-10}
\sum_{n=1}^\infty\frac{q^n-q^{3n}-q^{7n}+q^{9n}}{1-q^{10n}}\sin 2nz
=\frac{\eta^2(2\tau)\theta_1(2z|10\tau)\theta_4(z|10\tau)}{2\eta(10\tau)\theta_4(z|2\tau)}.
\end{align}
\end{cor}

\begin{proof}
Taking $N=10$ and $\chi=\psi$ in Theorem \ref{g-transform}, respectively,  we have
\begin{align}\label{eq:3.3}
&4\sum_{n=1}^\infty\frac{q^n-q^{3n}-q^{7n}+q^{9n}}{1-q^{10n}}\sin2nz \nonumber\\
=&\frac{\theta'_1}{\theta_1}(z-\pi\tau|10\tau)+\frac{\theta'_1}{\theta_1}(z+\pi\tau|10\tau)
-\frac{\theta'_1}{\theta_1}(z-3\pi\tau|10\tau)-\frac{\theta'_1}{\theta_1}(z+3\pi\tau|10\tau).
\end{align}
Replacing $\tau$ by $10\tau$ in the above equation and then letting $x_1$ to $z-\pi\tau$ , $x_2$ to $z+\pi\tau$ and $x_3$ to $-z+3\pi\tau$ in \eqref{four-variable}, we obtain
 \begin{align}\label{eq:3.4}
&\frac{\theta'_1}{\theta_1}(z-\pi\tau|10\tau)+\frac{\theta'_1}{\theta_1}(z+\pi\tau|10\tau)
-\frac{\theta'_1}{\theta_1}(z-3\pi\tau|10\tau)-\frac{\theta'_1}{\theta_1}(z+3\pi\tau|10\tau)\nonumber\\
=&\frac{-\theta_1^{'}(0|10\tau)\theta_1(2z|10\tau)\theta_1(2\pi\tau|10\tau)\theta_1(4\pi\tau|10\tau)}
{\theta_1(z-\pi\tau|10\tau)\theta_1(z+\pi\tau|10\tau)\theta_1(z-3\pi\tau|10\tau)\theta_1(z+3\pi\tau|10\tau)}.
\end{align}
Replacing $\tau$ by $10\tau$ in \eqref{theta1-drive-0}, we find that
\begin{align}\label{eq:3.5}
\theta'_1(0|10\tau)=2q^{5/4}(q^{10};q^{10})^3.
\end{align}
Replacing $\tau$ by $2\tau$ in \cite[Eq.\ (2.16)]{Liu04}, we obtain
\begin{align}\label{eq:3.6}
\theta_1(2\pi\tau|10\tau)\theta_1(4\pi\tau|10\tau)=-\eta(2\tau)\eta(10\tau).
\end{align}
From \eqref{theta1-infty-1} by some elementary calculation, it is easy to show that
\begin{align}\label{eq:3.7}
\theta_1(z-\pi\tau|10\tau)\theta_1(z+\pi\tau|10\tau)\theta_1(z-3\pi\tau|10\tau)\theta_1(z+3\pi\tau|10\tau)
=\frac{q(q^{10};q^{10})^5_\infty\theta_4(z|2\tau)}{(q^2;q^2)_\infty\theta_4(z|10\tau)}.
\end{align}
Combining the above equation \eqref{eq:3.3}-\eqref{eq:3.7}, we obtain \eqref{Lamber-10}.
\end{proof}

By  \eqref{Lamber-10} we obtain
\begin{itemize}
\item[(1)]
\begin{align}\label{Lamber-10-drive}
\sum_{n=1}^{\infty}\frac{n(q^n-q^{3n}-q^{7n}+q^{9n})}{1-q^{10n}}
=\frac{q(q^2;q^2)^3_\infty(q^5;q^5)^2_\infty(q^{10};q^{10})_\infty}{(q;q)^2_\infty},
\end{align}
or
\begin{align*}
\sum_{n=1}^{\infty}\psi(n)\frac{q^n}{(1-q^n)^2}
=\frac{q(q^2;q^2)^3_\infty(q^5;q^5)^2_\infty(q^{10};q^{10})_\infty}{(q;q)^2_\infty}.
\end{align*}

\begin{proof}
Dividing both sides of \eqref{Lamber-10} by $z$ and then letting $z\rightarrow 0$ , we are led to
\begin{align}\label{eq:3.8}
8\sum_{n=1}^\infty\frac{n(q^n-q^{3n}-q^{7n}+q^{9n})}{1-q^{10n}}
=\frac{2\eta^2(2\tau)\theta'_1(0|10\tau)\theta_4(0|10\tau)}{\eta(10\tau)\theta_4(0|2\tau)},
\end{align}
and
\begin{align}\label{eq:3.9}
\theta_4(0|2\tau)=\frac{\eta^2(\tau)}{\eta(2\tau)}.
\end{align}
The \eqref{Lamber-10-drive} follows after substituting \eqref{eq:3.5} and \eqref{eq:3.9} into \eqref{eq:3.8}.
\end{proof}

\item[(2)] Taking $z=\frac\pi3 $ and $z=\frac\pi4$ in  \eqref{Lamber-10}, respectively, we obtain identities
\begin{align*}
&\sum_{n=1}^{\infty}\bigg(\frac n3\bigg)\frac{q^n-q^{3n}-q^{7n}+q^{9n}}{1-q^{10n}}
=\frac{q(q;q)_\infty(q^6;q^6)_\infty(q^{10};q^{10})_\infty(q^{15};q^{15})_\infty}{(q^3;q^3)_\infty(q^5;q^5)_\infty},\\
&\sum_{n=1}^{\infty}\bigg(\frac{-4}{n}\bigg)\frac{q^n-q^{3n}-q^{7n}+q^{9n}}{1-q^{10n}}
=\frac{q(q^2;q^2)^2_\infty(q^{8};q^{8})_\infty(q^{20};q^{20})^4_\infty}{(q^{4};q^{4})^2_\infty(q^{10};q^{10})^2_\infty(q^{40};q^{40})_\infty}.
\end{align*}
\end{itemize}

Since $\psi$ is not primitive, the following identity will be derived via Theorem \ref{g-transform-imagy0}.

\begin{cor} \label{cor-Lamber-imagy-10}There holds the identity
\begin{align}\label{Lamber-imagy-10}
\frac{\sin2z\cos z}{\cos 5z}-\sum_{n=1}^\infty\psi(n)\frac{q^n}{1-q^n}\sin2nz
-\sum_{n=1,5\nmid n}^\infty\frac{(-1)^nq^{2n}}{1-q^{2n}}\sin4nz
=\frac{\eta^2(5\tau)\theta_1(2z|\tau)\theta_2(z|\tau)}{2\eta(\tau)\theta_2(5z|5\tau)}.
\end{align}
\end{cor}

\begin{proof} Taking $N=10$ and $\chi=\psi_{10}$ in Theorem \ref{g-transform-imagy0} and \eqref{Lamber-10}, respectively, we have
\begin{align*}
\sum_{k=1}^9\psi(k)\cot\bigg(z-\frac{k\pi}{10}\bigg)
+4\sum_{n=1}^{\infty}g_n(\psi)\frac{q^n}{1-q^n}\sin2nz
=\frac{\eta^2(-\frac{1}{5\tau})\theta_1\left(\frac{2z}{\tau}|-\frac{1}{\tau}\right)\theta_4\left(\frac z\tau|
-\frac 1\tau\right)}{2\eta\left(-\frac 1\tau\right)\theta_4\left(\frac z\tau|-\frac {1}{5\tau}\right)}.
\end{align*}
Together with
\begin{align*}
&\theta_1\bigg(\frac{2z}{\tau}\bigg|-\frac{1}{\tau}\bigg)
=- i\sqrt{- i\tau}e^{\frac{4 iz^2}{\pi\tau}}\theta_1(2z|\tau),
\quad\quad\theta_4\bigg(\frac{z}{\tau}\bigg|-\frac{1}{\tau}\bigg)
=\sqrt{- i\tau}e^{\frac{ iz^2}{\pi\tau}}\theta_2(z|\tau),\\
&\theta_4\bigg(\frac{z}{\tau}\bigg|-\frac{1}{5\tau}\bigg)
=\sqrt{-5 i\tau}e^{\frac{5 iz^2}{\pi\tau}}\theta_2(5z|5\tau),
\quad\quad\quad\eta\bigg(-\frac{1}{5\tau}\bigg)=\sqrt{-5 i\tau}\eta(5\tau),
\end{align*}
we derive
\begin{align*}
\sum_{k=1}^{9}\psi(k)\cot\left(z-\frac{k\pi}{10}\right)+4\sum_{n=1}^{\infty}g_n(\psi)\frac{q^n}{1-q^n}\sin2nz
=-2\sqrt5\frac{\eta^2(5\tau)\theta_1(2z|\tau)\theta_2(z|\tau)}{\eta(\tau)\theta_2(5z|5\tau)}.
\end{align*}
Letting $q\rightarrow0$, we deduce
\begin{align}\label{cos-10}
\sum_{k=1}^9\psi(k)\cot\bigg(z-\frac{k\pi}{10}\bigg)
=-4\sqrt5\frac{\sin 2z \cos z}{\cos 5z}.
\end{align}

We now compute
\begin{align*}
g_{n}(\psi)&=\sum_{k=1}^{9}\psi(k)e^{2 in\pi k/{10}}.
\end{align*}

It is easy to find that
\begin{align*}
\cos\frac{2\pi}{5}=\frac{\sqrt5-1}{4} \quad and \quad \cos\frac{\pi}{5}=\frac{\sqrt5+1}{4}.
\end{align*}
Then
\begin{align*}
g_1(\psi)&=e^{2i\pi/10}-e^{6i\pi/10}-e^{14i\pi/10}+e^{18i\pi/10}\\
&=2\cos\frac{\pi}{5}-2\cos\frac{3\pi}{5}=2\cos\frac{\pi}{5}+2\cos\frac{2\pi}{5}\\
&=\sqrt 5.
\end{align*}

For $(n, 10)=1$, we have
\begin{align*}
g_n(\chi)=\sum_{k=1}^{9}\chi(k)e^{2i\pi nk/10}=\chi(n)\sum_{k=1}^{9}\chi(k)e^{2i\pi k/10}
=\sqrt 5 \chi(n).
\end{align*}
Direct computation yields
\begin{align*}
&g_2(\chi)=-\sqrt 5,\quad\quad
g_4(\chi)=\sqrt 5,\quad\quad
g_5(\chi)=0,\\
&g_6(\chi)=-\sqrt 5,\quad\quad
g_8(\chi)=\sqrt 5.
\end{align*}

Hence
\begin{align*}
g_{n}(\psi)&=\sum_{k=1}^{9}\psi(k)e^{2 in\pi k/{10}}
=\begin{cases}\sqrt5&\text{if}\quad $n=1, 4, 8, 9$\\
 -\sqrt5&\text{if}\quad $n=2, 3, 6, 7$\\
 0&\text{if}\quad $n=0, 5$.
\end{cases}
\end{align*}

Collecting all above facts together, we derive the equation \eqref{Lamber-imagy-10}.
\end{proof}

By \eqref{Lamber-imagy-10} we can obtain
\begin{itemize}
\item[(1)]
\begin{align}\label{Lamber-imagy-10-drive}
1-\sum_{n=1}^\infty\psi(n)\frac{nq^n}{1-q^n}
-2\sum_{n=1,5\nmid n}^{\infty}\frac{(-1)^n nq^{2n}}{1-q^{2n}}
=\frac{(q;q)_\infty(q^2;q^2)^2_\infty(q^5;q^5)^3_\infty}{(q^{10};q^{10})^2_\infty}.
\end{align}

\begin{proof}
Dividing both sides of \eqref{Lamber-imagy-10} by $z$ and then letting $z\rightarrow 0$,
we obtain
\begin{align}\label{eq:4.10}
1-\sum_{n=1}^\infty\psi(n)\frac{nq^n}{1-q^n}
-2\sum_{n=1,5\nmid n}^{\infty}\frac{(-1)^n nq^{2n}}{1-q^{2n}}
=\frac{\eta^2(5\tau)\theta'_1(0|\tau)\theta_2(0|\tau)}{2\eta(\tau)\theta_2(0|5\tau)}.
\end{align}
By some straightforward calculations, we immediately deduce that
\begin{align}
\label{eq:4.11}\theta_2(0|\tau)&=2q^{1/8}\frac{(q^2;q^2)^2_\infty}{(q;q)_\infty},
\end{align}
and
\begin{align}
\label{eq:4.12}\theta_2(0|5\tau)&=2q^{5/8}\frac{(q^{10};q^{10})^2_\infty}{(q^5;q^5)_\infty}.
\end{align}
Now using \eqref{eq:4.11} and \eqref{eq:4.12} together with \eqref{theta1-drive-0} and \eqref{eq:4.10}, we complete the proof of \eqref{Lamber-imagy-10-drive}.
\end{proof}

\item[(2)] Taking $z=\frac\pi3$ and $z=\frac\pi4$ in  \eqref{Lamber-imagy-10}, respectively, we obtain identity
\begin{align*}
&1-\sum_{n=1}^\infty\psi(n)\bigg(\frac{n}{3}\bigg)\frac{q^n}{1-q^n}
-\frac {2}{\sqrt3}\sum_{n=1,5\nmid n}^{\infty}\sin\frac{n\pi}{3}\frac{q^{2n}}{1-q^{2n}}
=\frac{(q;q)_\infty(q^6;q^6)_\infty(q^{10};q^{10})_\infty(q^{15};q^{15})_\infty}{(q^2;q^2)_\infty(q^{30};q^{30})_\infty},\\
&1+\sum_{n=1}^\infty\psi(n)\bigg(\frac{-4}{n}\bigg)\frac{q^n}{1-q^n}
=\frac{(q^2;q^2)_\infty(q^4;q^4)_\infty(q^5;q^5)_\infty(q^{10};q^{10})_\infty}{(q;q)_\infty(q^{20};q^{20})_\infty}.
\end{align*}
\end{itemize}

\subsection{Representations for d=12}
Taking $N=12$ in Theorem \ref{g-transform} and $\chi(n)=\chi_{12}(n)=\left(\frac{12}{n}\right)$, we derive in a identical fashion as for the case $N=8$, the following identities:

\begin{itemize}
\item[(1)]
\begin{align*}
\sum_{n=1}^\infty\frac{q^n-q^{5n}-q^{7n}+q^{11n}}{1-q^{12n}}\sin2nz
=\frac{\eta(2\tau)\eta(4\tau)\eta^3(6\tau)\theta_1(z|3\tau)\theta_1(2z|12\tau)}{2\eta(3\tau)\eta^2(12\tau)\theta_1(z|6\tau)\theta_4(z|2\tau)},
\end{align*}


\item[(2)] 
\begin{align*}
\sum_{n=1}^{\infty}\frac{n(q^n-q^{5n}-q^{7n}+q^{11n})}{1-q^{12n}}
&=\frac{q(q^2;q^2)^2_\infty(q^3;q^3)^2_\infty(q^{4};q^{4})_\infty(q^{12};q^{12})_\infty}{(q;q)^2_\infty},
\end{align*}
\item[(3)]
\begin{align*}
\sum_{n=1}^{\infty}\left(\frac n3\right)\frac{q^n-q^{5n}-q^{7n}+q^{11n}}{1-q^{12n}}
=\frac{q(q;q)_\infty(q^{4};q^{4})_\infty(q^{6};q^{6})^4_\infty(q^{9};q^{9})_\infty(q^{36};q^{36})_\infty}
{(q^{2};q^{2})_\infty(q^{3};q^{3})^2_\infty(q^{12};q^{12})^2_\infty(q^{18};q^{18})_\infty},\\
\sum_{n=1}^{\infty}\left(\frac {-4}{n}\right)\frac{q^n-q^{5n}-q^{7n}+q^{11n}}{1-q^{12n}}
=\frac{q(q;q)_\infty(q^{2};q^{2})_\infty(q^{6};q^{6})_\infty(q^{8};q^{8})_\infty(q^{24};q^{24})_\infty}
      {(q^{4};q^{4})_\infty(q^{12};q^{12})_\infty},
\end{align*}

\item[(4)]
\begin{align*}
&\frac{\sin4z}{\cos6z}-2\sum_{n=1}^\infty\bigg(\frac{12}{n}\bigg)\frac{q^n}{1-q^n}\sin2nz
=\frac{\eta^3(2\tau)\eta(3\tau)\eta(6\tau)\theta_1(4z|4\tau)\theta_1(2z|\tau)}
{\eta^2(\tau)\eta(4\tau)\theta_1(2z|2\tau)\theta_2(6z|6\tau)},
\end{align*}

\item[(5)]
\begin{align*}
1-\sum_{n=1}^\infty\bigg(\frac{12}{n}\bigg)\frac{nq^n}{1-q^n}
=\frac{(q;q)_\infty(q^3;q^3)_\infty(q^4;q^4)^2_\infty(q^6;q^6)^2_\infty}{(q^{12};q^{12})^2_\infty},
\end{align*}

\item[(6)]\label{Lamber-imagy-12}
\begin{align}\label{Lamber-imagy-12-pi3}
&1+2\sum_{n=1}^{\infty}\bigg(\frac {-36}{n}\bigg)\frac{q^n}{1-q^n}
=\frac{(q^2;q^2)^3_\infty(q^3;q^3)^2_\infty(q^{6};q^{6})_\infty}{(q;q)^2_\infty(q^4;q^4)_\infty(q^{12};q^{12})_\infty},\\
&\sum_{n=1}^{\infty}\bigg(\frac{-4}{n}\bigg)\frac{q^n-q^{5n}-q^{7n}+q^{11n}}{1-q^{12n}}
=\frac{q(q^4;q^4)^2_\infty(q^{6};q^{6})_\infty(q^{12};q^{12})^3_\infty}{(q^{2};q^{2})_\infty(q^{24};q^{24})_\infty}.
\end{align}
\end{itemize}

\subsection{Representations for d=5}
Taking $N=5$ in Theorem \ref{g-transform} and $\chi(n)=\chi_5(n)=\left(\frac 5n\right)$, we have
\begin{itemize}
\item[(1)]
\begin{align*} 
\sum_{n=1}^\infty\frac{q^n-q^{2n}-q^{3n}+q^{4n}}{1-q^{5n}}\sin 2nz
=-\frac{q^{-1/8}(q;q)^2_\infty}{2(q^5;q^5)_\infty}\frac{\theta_1(z|5\tau)\theta_1(2z|5\tau)}{\theta_1(z|\tau)},
\end{align*}
\item[(2)]
\begin{align}\label{q-5-1}
\sum_{n=1}^{\infty}\left(\frac n5\right)\frac{q^n}{(1-q^n)^2}=\frac{q(q^5;q^5)^5_\infty}{(q;q)_\infty},
\end{align}
\item[(3)]
\begin{align*}
&\sum_{n=1}^{\infty}\left(\frac n3\right)\frac{q^n-q^{2n}-q^{3n}+q^{4n}}{1-q^{5n}}
=\frac{q(q;q)^2_\infty(q^{15};q^{15})^2_\infty}{(q^{3};q^{3})_\infty(q^{5};q^{5})_\infty},\\
&\sum_{n=1}^{\infty}\left(\frac {-4}{n}\right)\frac{q^n-q^{2n}-q^{3n}+q^{4n}}{1-q^{5n}}
=\frac{q(q;q)_\infty(q^{2};q^{2})_\infty(q^{10};q^{10})_\infty(q^{20};q^{20})_\infty}
      {2(q^{4};q^{4})_\infty(q^{5};q^{5})_\infty},
\end{align*}

\item[(4)]
\begin{align}\label{Lamber-imagy-5}
\frac{\sin z\sin 2z}{\sin 5z}-\sum_{n=1}^{\infty}\left(\frac n5\right)\frac{q^n}{1-q^n}\sin 2nz
=q^{3/8}\frac{(q^5;q^5)^2_\infty\theta_1(z|\tau)\theta_1(2z|\tau)}{2(q;q)_\infty\theta_1(5z|5\tau)},
\end{align}

\item[(5)]
\begin{align}\label{q-1-5}
&1-5\sum_{n=1}^{\infty}\left(\frac n5\right)\frac{q^n}{1-q^n}=\frac{(q;q)^5_\infty}{(q^5;q^5)_\infty},
\end{align}
\item[(6)]
\begin{align}\label{Lamber-imagy-5-pi3}
&1+\sum_{n=1}^{\infty}\bigg(\frac {-15}{n}\bigg)\frac{q^n}{1-q^n}
=\frac{(q^3;q^3)^2_\infty(q^5;q^5)^2_\infty}{(q;q)_\infty(q^{15};q^{15})_\infty},\\
\label{Lamber-imagy-5-pi4}&1+\sum_{n=1}^{\infty}\bigg(\frac{-20}{n}\bigg)\frac{q^n}{1-q^{n}}
=\frac{(q^2;q^2)_\infty(q^{4};q^{4})_\infty(q^{5};q^{5})_\infty(q^{10};q^{10})_\infty}{(q;q)_\infty(q^{20};q^{20})_\infty}.
\end{align}
\end{itemize}

\begin{remark}
 The identities \eqref{q-5-1} and \eqref{q-1-5} can also be obatined in \cite[p. 139-140]{Ram1988} and \cite[Theorem 5.13]{Cooper}. In \cite[Proposition 5.1]{Liu12}, we can also find  the identity \eqref{Lamber-imagy-5}.
\end{remark}

\begin{remark}
Let $\mathbb F$ be the imaginary quadratic field of the discriminant $d$. It is known that
\begin{align*}
\sum_{m,n=-\infty}^{\infty} \sum_{i=1}^h q^{Q_i(m,n)}=h+w \sum_{n=1}^{\infty}\bigg(\frac{d}n\bigg)\frac{q^n}{1-q^{n}};
\end{align*}
where $h$ is the class number of $\mathbb F$ and $Q_i(x,y) (i=1,2,..,h)$ is the set of inequivalent quadratic forms of discriminant $d$ and $w$ is the number of units in $\mathbb F$.\\

For $d=-24$, $h=2$, $w=2$,
 \begin{align*}
Q_1(x,y)=x^2+6y^2, \quad  Q_2(x,y)=2x^2+3y^2
 \end{align*}
 and the Lambert series in \eqref{Lamber-imagy-12-pi3} can be represented as
 \begin{align}\label{1+6}
\sum_{m,n=-\infty}^{\infty} q^{m^2+6n^2}+ q^{2m^2+3n^2} =&2+2\sum_{n=1}^{\infty}\bigg(\frac{-24}{n}\bigg)\frac{q^n}{1-q^{n}}\nonumber\\
=&2\frac{(q^2;q^2)_\infty(q^3;q^3)_\infty(q^8;q^8)_\infty(q^{12};q^{12})_\infty}{(q;q)_\infty(q^{24};q^{24})_\infty}.
\end{align}
Similarly, for $d=-15$, $h=2$, $w=2$,
 \begin{align*}
Q_1(x,y)=x^2+xy+4y^2, \quad  Q_2(x,y)=2x^2+xy+2y^2
 \end{align*}
and the Lambert series in \eqref{Lamber-imagy-5-pi3} can be represented as
 \begin{align*}
\sum_{m,n=-\infty}^{\infty} q^{m^2+mn+4n^2}+ q^{2m^2+mn+2n^2} =2+2\sum_{n=1}^{\infty}\bigg(\frac{-15}{n}\bigg)\frac{q^n}{1-q^{n}}
=2\frac{(q^3;q^3)^2_\infty(q^5;q^5)^2_\infty}{(q;q)_\infty(q^{15};q^{15})_\infty};
 \end{align*}
and for $d=-20$, $h=2$, $w=2$,
 \begin{align*}
Q_1(x,y)=x^2+5y^2, \quad  Q_2(x,y)=2x^2+2xy+3y^2
 \end{align*}
and the Lambert series in \eqref{Lamber-imagy-5-pi4} can be represented as
 \begin{align}\label{1+5}
\sum_{m,n=-\infty}^{\infty} q^{m^2+5n^2}+ q^{2m^2+2mn+3n^2} =&2+2\sum_{n=1}^{\infty}\bigg(\frac{-20}{n}\bigg)\frac{q^n}{1-q^{n}}\nonumber\\
=&2\frac{(q^2;q^2)_\infty(q^4;q^4)_\infty(q^5;q^5)_\infty(q^{10};q^{10})_\infty}{(q;q)_\infty(q^{20};q^{20})_\infty}.
\end{align}

Note that \eqref{1+6} and \eqref{1+5} can also be found in \cite[p. 385, p. 380]{Berk}.
\end{remark}

\section{Acknowledgements}
We would like to thank Professor Frank Garvan for bringing  our attention to the work of Professor Kolberg. We would like to thank Professor Zhi-Guo Liu and Professor Li-Chien Shen for their patient guidance. The authors were supported in part by the National Natural Science Foundation of China (Grant No. 11971173 ) and ECNU Short-term Overseas Research Scholarship for Graduate Students (Grant No. 201811280046 and No. 201811280047 ).

\end{document}